\documentclass[11pt,a4paper]{amsart}

\usepackage{amsmath}
\usepackage{mathtools}
\usepackage{amsthm}
\usepackage{amssymb}
\usepackage{hyperref}

\usepackage{pdflscape}
\usepackage{afterpage}
\usepackage{booktabs}

\newcommand*\fullref[3][\relax]{%
  \ifdefined\hyperref%
    {\hyperref[#3]{#2\penalty 200\ \ref*{#3}#1}}%
  \else%
    {#2\penalty 200\ \relax\ref{#3}#1}%
  \fi%
}

\usepackage{tikz}
\usetikzlibrary{calc}
\usetikzlibrary{shapes.geometric} % For "isoceles triangle"
\usetikzlibrary{shapes.misc}  % For "rounded rectangle"

\tikzset{
  bst/.style={
    standard/.style={
      font=\small,
      draw=gray,
      rounded rectangle,
      minimum width=4.5mm,
      minimum height=4.5mm,
      inner xsep=0mm,
      inner ysep=1mm,
      outer sep=0mm,
      line width=.5pt,
    },
    empty/.style={
      minimum width=3mm,
      minimum height=3mm,
    },
    triangle/.style={
      isosceles triangle,
      isosceles triangle apex angle=60,
      shape border rotate=90,
      rounded corners=2mm,
      minimum width=8mm,
      inner xsep=0mm,
      inner ysep=.5mm
    },
    blank/.style={
      draw=none,
    },
    nodecount/.style={
      blank,
      font=\scriptsize,
    },
    every node/.style={standard},
    every child/.style={draw=black,line width=.6pt},
    level distance=10mm,
    level 1/.style={sibling distance=60mm},
    level 2/.style={sibling distance=30mm},
    level 3/.style={sibling distance=15mm},
  },
  medbst/.style={
    bst,
    level distance=10mm,
    level 1/.style={sibling distance=15mm},
    level 2/.style={sibling distance=15mm},
    level 3/.style={sibling distance=15mm},
  },
  smallbst/.style={
    bst,
    level distance=8mm,
    level 1/.style={sibling distance=10mm},
    level 2/.style={sibling distance=10mm},
    level 3/.style={sibling distance=10mm},
  },
  tinybst/.style={
    bst,
    level distance=5mm,
    level 1/.style={sibling distance=8mm},
    level 2/.style={sibling distance=8mm},
    level 3/.style={sibling distance=8mm},
    every node/.append style={
      font=\footnotesize,
    },
    triangle/.append style={
      rounded corners=1mm,
      minimum width=7mm,
      inner xsep=-.5mm,
    },
  },
  microbst/.style={
    bst,
    standard/.append style={
      font=\scriptsize,
      minimum width=3mm,
      minimum height=3mm,
      inner ysep=.25mm,
    },
    level distance=3mm,
    level 1/.style={sibling distance=6mm},
    level 2/.style={sibling distance=6mm},
    level 3/.style={sibling distance=6mm},
  },
  nanobst/.style={
    bst,
    standard/.append style={
      font=\tiny,
      minimum width=2mm,
      minimum height=2mm,
      inner ysep=.25mm,
    },
    level distance=2mm,
    level 1/.style={sibling distance=4mm},
    level 2/.style={sibling distance=4mm},
    level 3/.style={sibling distance=4mm},
  },
}
\usetikzlibrary{matrix}

\tikzset{
  pretableaumatrix/.style={
    ampersand replacement=\&,
    matrix of math nodes,
    outer sep=1mm,
    inner sep=0mm,
    anchor=center,
    row sep={between borders,-\pgflinewidth},
    column sep={between borders,-\pgflinewidth},
    dottedentry/.style={densely dotted},
    spaceentry/.style={draw=none,execute at begin node=\null},
  },
  pretableaunode/.style={
    font=\small,
    draw=gray,
    sharp corners,
    rectangle,
    anchor=base,
    text height=3.75mm,
    text depth=1.25mm,
    minimum height=5mm,
    minimum width=5mm,
    inner sep=0mm,
    outer sep=0mm,
  },
  tableaumatrix/.style={
    pretableaumatrix,
    every node/.append style={
      pretableaunode,
    },
  },
  medtableaumatrix/.style={
    pretableaumatrix,
    every node/.append style={
      pretableaunode,
      font=\footnotesize,
      text height=2.75mm,
      text depth=.75mm,
      minimum height=3.5mm,
      minimum width=3.5mm
    },
  },
  smalltableaumatrix/.style={
    pretableaumatrix,
    every node/.append style={
      pretableaunode,
      font=\scriptsize,
      text height=1.85mm,
      text depth=.15mm,
      minimum height=2.5mm,
      minimum width=2.5mm,
    },
  },
  tinytableaumatrix/.style={
    pretableaumatrix,
    every node/.append style={
      pretableaunode,
      font=\tiny,
      text height=1.25mm,
      text depth=.15mm,
      minimum height=1.75mm,
      minimum width=1.75mm
    },
  },
  tableau/.style={
    baseline=-1.25mm,
    every matrix/.style={tableaumatrix},
  },
  medtableau/.style={
    baseline=-1.25mm,
    every matrix/.style={medtableaumatrix},
  },
  smalltableau/.style={
    baseline=-1.25mm,
    every matrix/.style={smalltableaumatrix},
  },
  preshapetableaumatrix/.style={
    pretableaumatrix,
    execute at end cell={\strut},
    every node/.append style={
      draw=black,
      anchor=base,
      inner sep=0mm,
      outer sep=0mm,
    },
    shadedentry/.style={fill=gray},
    darkshadedentry/.style={fill=darkgray},
  },
  medshapetableaumatrix/.style={
    preshapetableaumatrix,
    every node/.append style={
      text height=2.75mm,
      text depth=.75mm,
      minimum height=3.5mm,
      minimum width=3.5mm
    },
  },
  shapetableaumatrix/.style={
    ampersand replacement=\&,
    matrix of math nodes,
    outer sep=0mm,
    inner sep=0mm,
    anchor=base,
    row sep={between borders,-\pgflinewidth},
    column sep={between borders,-\pgflinewidth},
    execute at begin cell={\strut},
    every node/.append style={draw,anchor=base,text height=1mm,text depth=.5mm,minimum size=1.5mm,inner sep=0mm,outer sep=0mm},
  },
  shapetableau/.style={
    every matrix/.style={shapetableaumatrix},
  },
  topalign/.style={
    every matrix/.append style={name=maintableau,anchor=maintableau-1-1.base},
    baseline,
  },
}

\newcommand*\tableau[2][]{\tikz[tableau,#1]\matrix{#2};}

\theoremstyle{definition}
\newtheorem{definition}{Definition}[section]

\newtheorem{algorithm}[definition]{Algorithm}
\newtheorem{conjecture}[definition]{Conjecture}

\theoremstyle{plain}

\newtheorem{lemma}[definition]{Lemma}
\newtheorem{proposition}[definition]{Proposition}

\numberwithin{equation}{section}
\newcommand*{\textparens}[1]{\textup{(}#1\textup{)}}

\newcommand*{\defterm}[1]{\emph{#1}}

\makeatletter
\newcommand\chyph{\penalty\@M-\hskip\z@skip}
\makeatother

\DeclarePairedDelimiter{\parens}{\lparen}{\rparen}

\DeclarePairedDelimiter{\set}{\{}{\}}
\DeclarePairedDelimiterX{\gset}[2]{\{}{\}}{\,#1:#2\,}
\newcommand*{\biggg}{\bBigg@{4}}

\newcommand*{\Biggg}{\bBigg@{5}}

\makeatletter
\newcommand*{\sizeddelimiter}[2]{\bBigg@{#1}#2}
\makeatother

\newcommand*{\emptyword}{\varepsilon}

\DeclarePairedDelimiterX{\pres}[2]{\langle}{\rangle}{#1\,\delimsize\vert\,\mathopen{}#2}

\newcommand*{\aA}{\mathcal{A}}

\newcommand*{\evlit}{{\mathrm{ev}}}
\newcommand*{\ev}[2][]{\evlit\parens[#1]{#2}}

\newcommand*{\plac}{{\mathsf{plac}}}
\newcommand*{\hypo}{{\mathsf{hypo}}}
\newcommand*{\sylv}{{\mathsf{sylv}}}
\newcommand*{\sylvsharp}{{{\mathsf{sylv}}\smash{{}^{\mathsf{\#}}}}}
\newcommand*{\baxt}{{\mathsf{baxt}}}
\newcommand*{\stal}{{\mathsf{stal}}}
\newcommand*{\taig}{{\mathsf{taig}}}

\newcommand*{\lps}{{\mathsf{lPS}}}
\newcommand*{\rps}{{\mathsf{rPS}}}

\newcommand*{\plit}{\mathrm{P}}

\newcommand*{\pplac}[2][]{\plit_{\plac}\parens[#1]{#2}}

\newcommand*{\phypo}[2][]{\plit_{\hypo}\parens[#1]{#2}}

\newcommand*{\psylv}[2][]{\plit_{\sylv}\parens[#1]{#2}}

\newcommand*{\psylvsharp}[2][]{\plit_{\sylvsharp}\parens[#1]{#2}}

\newcommand*{\pbaxt}[2][]{\plit_{\baxt}\parens[#1]{#2}}

\newcommand*{\ptaig}[2][]{\plit_{\taig}\parens[#1]{#2}}

\newcommand*{\pstal}[2][]{\plit_{\stal}\parens[#1]{#2}}

\newcommand*{\plps}[2][]{\plit_{\lps}\parens[#1]{#2}}

\newcommand*{\prps}[2][]{\plit_{\rps}\parens[#1]{#2}}

\begin{document}

\title[Identities in plactic, hypoplactic, and related monoids]{Identities in plactic, hypoplactic, sylvester, Baxter, and related monoids}

\author{Alan J. Cain}
\address{%
Centro de Matem\'{a}tica e Aplica\c{c}\~{o}es\\
Faculdade de Ci\^{e}ncias e Tecnologia\\
Universidade Nova de Lisboa\\
2829--516 Caparica\\
Portugal
}
\email{%
a.cain@fct.unl.pt
}
\thanks{The first author was supported by an Investigador {\sc FCT} fellowship ({\scshape IF}/01622/2013/{\scshape CP}1161/{\scshape
    CT}0001). For both authors, this work was partially supported by the
Funda\c{c}\~{a}o para
a Ci\^{e}ncia e a Tecnologia (Portuguese Foundation for Science and Technology) through the project {\scshape UID}/{\scshape
  MAT}/00297/2013 (Centro de Matem\'{a}tica e Aplica\c{c}\~{o}es), and the project {\scshape PTDC}/{\scshape MHC-FIL}/2583/2014.}

\author{Ant\'onio Malheiro}
\address{%
Departamento de Matem\'{a}tica \&\ Centro de Matem\'{a}tica e Aplica\c{c}\~{o}es\\
Faculdade de Ci\^{e}ncias e Tecnologia\\
Universidade Nova de Lisboa\\
2829--516 Caparica\\
Portugal
}
\email{%
ajm@fct.unl.pt
}

\begin{abstract}
  This paper considers whether non-trivial identities are satisfied by certain `plactic-like' monoids that, like the
  plactic monoid, are closely connected to combinatorics. New results show that the hypoplactic, sylvester, Baxter,
  stalactic, and taiga monoids satisfy identities, and indeed give shortest identities satisfied by these monoids. The
  existing state of knowledge is discussed for the plactic monoid and left and right patience sorting monoids.
\end{abstract}

\maketitle

\section{Introduction}

The ubiquitous plactic monoid, whose elements can be viewed as semistandard Young tableaux, and which appears in such
diverse contexts as symmetric functions \cite{macdonald_symmetric}, representation theory \cite{fulton_young}, algebraic
combinatorics \cite{lothaire_algebraic}, Kostka--Foulkes polynomials \cite{lascoux_plaxique,lascoux_foulkes}, Schubert
polynomials \cite{lascoux_schubert,lascoux_tableaux}, and musical theory \cite{jedrzejewski_plactic}, is one of a family
of `plactic-like' monoids that are closely connected with combinatorics. These monoids include the hypoplactic monoid
\cite{krob_noncommutative4,novelli_hypoplactic}, the sylvester monoid \cite{hivert_sylvester}, the taiga monoid
\cite{priez_lattice}, the stalactic monoid \cite{hivert_commutative,priez_lattice}, the Baxter monoid
\cite{giraudo_baxter2}, and the left and right patience sorting monoids \cite{rey_algebraic,cms_patience1}. Each of these
monoids is obtained by factoring the free monoid $\aA^*$ over the infinite ordered alphabet
$\aA = \set{1 < 2 < 3 <\ldots}$ by a congruence that arises from a so-called insertion algorithm that computes a
combinatorial object from a word. For instance, for the plactic monoid, the corresponding combinatorial objects are
(semistandard) Young tableaux; for the sylvester monoid, they are binary search trees.

\afterpage{%
\begin{landscape}
\begin{table}[ht]
  \centering
  \caption{Examples of non-trivial identities satisfied by `plactic-like' monoids. The stated identities are always
    shortest non-trivial identities satisfied by the corresponding monoid, but there may be other identities of the same
    length.}
  \label{tbl:summary}
  \begin{tabular}{llr@{\;}c@{\;}lll}
    \toprule
    \textit{Monoid}        & \textit{Symbol} &          & \clap{\textit{Identity satisfied}} &          & \emph{Discussed in}                    & \emph{Original result}     \\
    \midrule
    Plactic                & $\plac$         &          & None                               &          & \fullref{Subsec.}{subsec:placiden}     & \cite{ckkmo_placticidentity}    \\
    Hypoplactic            & $\hypo$         & $xyxy$   & $=$                                & $yxyx$   & \fullref{Subsec.}{subsec:hypoiden}     & Present paper       \\
    Sylvester              & $\sylv$         & $xyxy$   & $=$                                & $yxxy$   & \fullref{Subsec.}{subsec:sylviden}     & Present paper       \\
    $\#$-sylvester         & $\sylvsharp$    & $yxyx$   & $=$                                & $yxxy$   & \fullref{Subsec.}{subsec:sylviden}     & Present paper       \\
    Baxter                 & $\baxt$         & $yxxyxy$ & $=$                                & $yxyxxy$ & \fullref{Subsec.}{subsec:baxtiden}     & Present paper       \\
    Stalactic              & $\stal$         & $xyx$    & $=$                                & $yxx$    & \fullref{Subsec.}{subsec:staltaigiden} & Present paper       \\
    Taiga                  & $\taig$         & $xyx$    & $=$                                & $yxx$    & \fullref{Subsec.}{subsec:staltaigiden} & Present paper       \\
    Left patience sorting  & $\lps$          &          & \clap{None}                        &          & \fullref{Subsec.}{subsec:lpsrpsiden}   & \cite{cms_patience1} \\
    Right patience sorting & $\rps$          &          & \clap{None}                        &          & \fullref{Subsec.}{subsec:lpsrpsiden}   & \cite{cms_patience1} \\
    \bottomrule
  \end{tabular}
\end{table}
\end{landscape}
}

An \defterm{identity} is a formal equality between two words in the free monoid, and is \defterm{non-trivial} if the two
words are distinct. A monoid $M$ \defterm{satisfies} such an identity if the equality in $M$ holds under every
substitution of letters in the words by elements of $M$. For example, any commutative monoid satisfies the non-trivial
identity $xy = yx$. A finitely generated group has polynomial growth if and only if it is virtually nilpotent
\cite{gromov_growth}, and a virtually nilpotent group satisfies a non-trivial identity
\cite{malcev_nilpotent,neumann_nilpotent}. Thus it is natural to ask whether every finitely generated semigroup or
monoid with polynomial growth satisfies a non-trivial identity. Schneerson \cite{shneerson_identities} provided the
first counterexample, but it remains to be seen whether there is a `natural' finitely generated semigroup with
polynomial growth that does not satisfy a non-trivial identity. It is easy to see that the finite-rank analogues of the plactic
monoid and the other related monoids discussed above have polynomial growth. This naturally leads to the question of
whether these monoids satisfy non-trivial identities, for if any of them failed to do so, it would certainly be a very
natural example of a polynomial-growth monoid that does not satisfy a non-trivial identity.

Further motivation for this question comes from a result of Jaszu\'{n}ska \& Okni\'{n}ski, who proved that the Chinese
monoid \cite{cassaigne_chinese}, which has the same growth type as the plactic monoid \cite{duchamp_placticgrowth} but
which does not arise from such a natural combinatorial object, satisfies Adian's identity $xyyxxyxyyx = xyyxyxxyyx$
\cite[Corollary~3.3.4]{jaszunska_chinese}. (This is the shortest non-trivial identity satisfied by the bicyclic monoid
\cite[Chapter IV, Theorem 2]{adian_defining}).

The goal of this paper is to present new results showing that some of these monoids satisfy non-trivial identities, and
to survey the state of knowledge for other monoids in this family. New results show that the hypoplactic, sylvester,
baxter, stalactic, and taiga monoids satisfy non-trivial identities. A discussion of the situation for the left and right patience sorting monoids and
plactic monoids completes the paper. \fullref{Table}{tbl:summary} summarizes the results.

\section{`Plactic-like' monoids}

In this section, we recall only the definition and essential facts about the various monoids; for further background,
see \cite[Ch.~5]{lothaire_algebraic} on the plactic monoid, \cite{novelli_hypoplactic} on the hypoplactic monoid,
\cite{hivert_sylvester} on the sylvester monoid; \cite[\S~5]{priez_lattice} on the taiga monoid; \cite{priez_lattice} on
the stalactic monoid; \cite{giraudo_baxter2} on the Baxter monoid; and \cite{cms_patience1} on the patience sorting monoids.

\subsection{Alphabets and words}

For any alphabet $X$, the free monoid (that is, the set of all words, including the empty word) on the
alphabet $X$ is denoted $X^*$. The empty word is denoted $\emptyword$. For any $u \in X^*$, the length of $u$ is denoted
$|u|$, and, for any $x \in X$, the number of times the symbol $x$ appears in $u$ is denoted $|u|_x$.

Throughout the paper, $\aA = \set{1<2<3<\ldots}$ is the set of natural numbers viewed as an infinite ordered alphabet,
and $\aA_n = \set{1<2<\ldots<n}$ is set of the first $n$ natural numbers viewed as a finite ordered alphabet.

\subsection{Combinatorial objects and insertion algorithms}

A \defterm{Young tableau} is a finite array of symbols from $\aA$, with rows non-decreasing from left to right and columns
strictly increasing from top to bottom, with shorter rows below longer ones, and with rows left-justified. An
example of a Young tableau is
\begin{equation}
\label{eq:tableaueg}
\tableau{
1 \& 1 \& 1 \& 2 \& 5 \\
3 \& 3 \& 5 \& 6 \\
6 \\
}.
\end{equation}

The following algorithm takes a Young tableau and a symbol from $\aA$ and yields a new Young tableau:

\begin{algorithm}[Schensted's algorithm]
\label{alg:placinsertone}
\textit{Input:} A Young tableau $T$ and a symbol $a \in \aA$.
\begin{enumerate}

\item If $a$ is greater than or equal to every entry in the topmost row of $T$, add $a$ as an entry at the rightmost end of
  $T$ and output the resulting tableau.

\item Otherwise, let $z$ be the leftmost entry in the top row of $T$ that is strictly greater than $a$. Replace $z$ by
  $a$ in the topmost row and recursively insert $z$ into the tableau formed by the rows of $T$ below the topmost. (Note
  that the recursion may end with an insertion into an `empty row' below the existing rows of $T$.)

\end{enumerate}
\end{algorithm}

A \defterm{quasi-ribbon tableau} is a finite array of symbols from $\aA$,
with rows non-decreasing from left to right and columns strictly increasing from top to bottom, that does not contain
any $2\times 2$ subarray (that is, of the form $\tikz[shapetableau]\matrix{ \null \& \null \\ \null \& \null \\};$). An
example of a quasi-ribbon tableau is:
\begin{equation}
\label{eq:qrteg}
\tableau
{
 1 \& 1 \& 1 \& 2 \&   \&   \&   \&   \\
   \&   \&   \& 3 \& 3 \& 5 \& 5 \&   \\
   \&   \&   \&   \&   \&   \& 6 \& 6 \\
}.
\end{equation}
Notice that the same symbol cannot appear in two different rows of a quasi-ribbon tableau. There is also an insertion algorithm for quasi-ribon tableau:

\begin{algorithm}[{\cite[Algorithm~4.4]{novelli_hypoplactic}}]
\label{alg:hypoinsertone}
\textit{Input:} A quasi-ribbon tableau $T$ and a symbol $a \in \aA$.

If there is no entry in $T$ that is less than or equal to $a$, output the tableau obtained by putting
$a$ and gluing $T$ by its top-leftmost entry to the bottom of $a$.

Otherwise, let $x$ be the right-most and bottom-most entry of $T$ that is less than or equal to $a$. Put a new entry $a$
to the right of $x$ and glue the remaining part of $T$ (below and to the right of $x$) onto the bottom of the new entry
$a$. Output the new tableau.
\end{algorithm}

A \defterm{right strict binary search tree} is a labelled rooted binary tree where the label of each node is greater
than or equal to the label of every node in its left subtree, and strictly less than every node in its right subtree. An
example of a binary search tree is
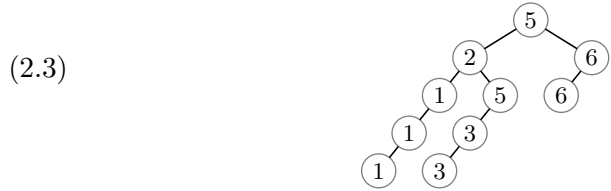
\begin{equation}
\label{eq:bsteg}
\begin{tikzpicture}[tinybst,baseline=-7.5mm]
  \node (root) {$5$}
    child[sibling distance=16mm] { node (0) {$2$}
      child { node (00) {$1$}
        child { node (000) {$1$}
          child { node (0000) {$1$} }
          child[missing]
        }
        child[missing]
      }
      child { node (01) {$5$}
        child { node (010) {$3$}
          child { node (0100) {$3$} }
          child[missing]
        }
        child[missing]
      }
    }
    child[sibling distance=16mm] { node (1) {$6$}
      child { node (10) {$6$} }
      child[missing]
    };
\end{tikzpicture}.
\end{equation}

The insertion algorithm for right strict binary search trees adds the new symbol as a leaf node in the unique place that maintains
the property of being a right strict binary search tree:

\begin{algorithm}[{\cite[\S~3.3]{hivert_sylvester}}]
\label{alg:sylvinsertone}
\textit{Input:} A right strict binary search tree $T$ and a symbol $a \in \aA$.

If $T$ is empty, create a node and label it $a$. If $T$ is non-empty, examine the label $x$ of the root
node; if $a \geq x$, recursively insert $a$ into the right subtree of the root node; otherwise recursively insert $a$
into the left subtree of the root note. Output the resulting tree.
\end{algorithm}

A \defterm{left strict binary search tree} is a labelled rooted binary tree where the label of each node is strictly
greater than the label of every node in its left subtree, and less than or equal to every node in its right subtree; see
the left tree shown in \eqref{eq:twinbsteg} below for an example.

The insertion algorithm for left strict binary search trees is dual to \fullref{Algorithm}{alg:sylvinsertone} above:

\begin{algorithm}
\label{alg:sylvsharpinsertone}
\textit{Input:} A left strict binary search tree $T$ and a symbol $a \in \aA$.

If $T$ is empty, create a node and label it $a$. If $T$ is non-empty, examine the label $x$ of the root
node; if $a \leq x$, recursively insert $a$ into the left subtree of the root node; otherwise recursively insert $a$
into the right subtree of the root note. Output the resulting tree.
\end{algorithm}

The \defterm{canopy} of a (labelled or unlabelled) binary tree $T$ is the word over $\set{0,1}$ obtained by traversing
the empty subtrees of the nodes of $T$ from left to right, except the first and the last, labelling an empty left
subtree by $1$ and an empty right subtree by $0$. (See \eqref{eq:twinbsteg} below for examples of canopies.)

A \defterm{pair of twin binary search trees} consist of a left strict binary search tree $T_L$ and a right strict binary
search tree $T_R$, such that $T_L$ and $T_R$ contain the same symbols, and the canopies of $T_L$ and $T_R$ are
complementary, in the sense that the $i$-th symbol of the canopy of $T_L$ is $0$ (respectively $1$) if and only if the
$i$-th symbol of the canopy of $T_R$ is $1$ (respectively $0$). The following is an example of a pair of twin binary
search trees, with the complementary canopies $0110101$ and $1001010$ shown in grey:
\begin{equation}
  \label{eq:twinbsteg}
  \left(
  \begin{tikzpicture}[baseline=-20mm]
    \begin{scope}[
      smallbst,
      ]
      \node (root) {$3$}
      child[sibling distance=25mm] { node (0) {$1$}
        child[missing]
        child { node (01) {$1$}
          child[missing]
          child { node (011) {$1$}
            child[missing]
            child { node (0111) {$2$} }
          }
        }
      }
      child[sibling distance=25mm] { node (1) {$6$}
        child { node (10) {$3$}
          child[missing]
          child { node (101) {$5$}
            child[missing]
            child { node (1011) {$5$} }
          }
        }
        child { node (11) {$6$} } };
    \end{scope}
    \begin{scope}[every node/.style={gray,font=\scriptsize}]
      \node at ($ (01) + (-2mm,-3.5mm) $) {$1$};
      \node at ($ (011) + (-2mm,-3.5mm) $) {$1$};
      \node at ($ (0111) + (-2mm,-3.5mm) $) {$1$};
      \node at ($ (0111) + (2mm,-3.5mm) $) {$0$};
      \node at ($ (10) + (-2mm,-3.5mm) $) {$1$};
      \node at ($ (101) + (-2mm,-3.5mm) $) {$1$};
      \node at ($ (1011) + (-2mm,-3.5mm) $) {$1$};
      \node at ($ (1011) + (2mm,-3.5mm) $) {$0$};
      \node at ($ (11) + (-2mm,-3.5mm) $) {$1$};
    \end{scope}
  \end{tikzpicture}
  \;\;
  ,
  \;
  \begin{tikzpicture}[baseline=-20mm]
    \begin{scope}[
      smallbst,
      ]
      \node (root) {$5$}
      child[sibling distance=25mm] { node (0) {$2$}
        child { node (00) {$1$}
          child { node (000) {$1$}
            child { node (0000) {$1$} }
            child[missing]
          }
          child[missing]
        }
        child { node (01) {$5$}
          child { node (010) {$3$}
            child { node (0100) {$3$} }
            child[missing]
          }
          child[missing]
        }
      }
      child[sibling distance=25mm] { node (1) {$6$}
        child { node (10) {$6$} }
        child[missing]
      };
    \end{scope}
    \begin{scope}[every node/.style={gray,font=\scriptsize}]
      \node at ($ (0000) + (2mm,-3.5mm) $) {$0$};
      \node at ($ (000) + (2mm,-3.5mm) $) {$0$};
      \node at ($ (00) + (2mm,-3.5mm) $) {$0$};
      \node at ($ (0100) + (-2mm,-3.5mm) $) {$1$};
      \node at ($ (0100) + (2mm,-3.5mm) $) {$0$};
      \node at ($ (010) + (2mm,-3.5mm) $) {$0$};
      \node at ($ (01) + (2mm,-3.5mm) $) {$0$};
      \node at ($ (10) + (-2mm,-3.5mm) $) {$1$};
      \node at ($ (10) + (2mm,-3.5mm) $) {$0$};
    \end{scope}
  \end{tikzpicture}
  \right).
\end{equation}

A \defterm{stalactic tableau} is a finite array of symbols of $\aA$ in which columns are top-aligned, and two symbols
appear in the same column if and only if they are equal. For example,
\begin{equation}
\label{eq:staltabeg}
\tikz[tableau]\matrix{
3 \& 1 \& 2 \& 6 \& 5 \\
3 \& 1 \&   \& 6 \& 5 \\
  \& 1 \&   \&   \&   \\
};
\end{equation}
is a stalactic tableau. The insertion algorithm is very straightforward:

\begin{algorithm}[{\cite[\S~3.7]{hivert_commutative}}]
\label{alg:stalinsertone}
\textit{Input:} A stalactic tableau $T$ and a symbol $a \in \aA$.

If $a$ does not appear in $T$, add $a$ to the left of the top row of $T$. If $a$ does appear in $T$,
add $a$ to the bottom of the (by definition, unique) column in which $a$ appears. Output the new tableau.
\end{algorithm}

A \defterm{binary search tree with multiplicities} is a labelled binary search tree in which each label appears at most
once, and where a non-negative integer called the \defterm{multiplicity} is assigned to each node label. An example of a
binary search tree is:
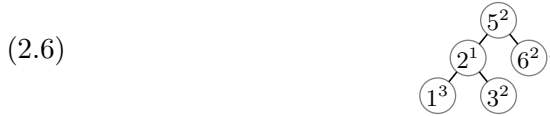
\begin{equation}
\label{eq:bstmulteg}
\begin{tikzpicture}[tinybst,baseline=(0)]
  \node (root) {$5^2$}
    child { node (0) {$2^1$}
      child { node (01) {$1^3$} }
      child { node (01) {$3^2$} }
    }
    child { node (1) {$6^2$} };
\end{tikzpicture}.
\end{equation}
(The superscripts on the labels in each node denote the multiplicities.)

\begin{algorithm}(\cite[Algorithm~3]{priez_lattice})
\label{alg:taiginsertone}
\textit{Input:} A binary search tree with multiplicities $T$ and a symbol $a \in \aA$.

If $T$ is empty, create a node, label it by $a$, and assign it multiplicity $1$. If $T$ is non-empty,
examine the label $x$ of the root node; if $a < x$, recursively insert $a$ into the left subtree of the root node; if $a
> x$, recursively insert $a$ into the right subtree of the root note; if $a=x$, increment by $1$ the multiplicity of the node
label $x$.
\end{algorithm}

An \defterm{lPS tableau} is a finite array of symbols of $\aA$ in which columns are top-aligned, the entries in the
top row are non-decreasing from left to right, and the entries in each column are strictly increasing from top to bottom. For example,
\begin{equation}
\label{eq:lpstabeg}
\tikz[tableau]\matrix{
1 \& 1 \& 1 \& 2 \& 5 \\
3 \& 3 \& 5 \&   \& 6 \\
  \& 6 \\
};
\end{equation}
is an lPS tableau. The insertion algorithm is very straightforward:

\begin{algorithm}[{\cite[\S~3.2]{thomas_longest}}]
\label{alg:lpsinsertone}
\textit{Input:} An lPS tableau $T$ and a symbol $a \in \aA$.

If $a$ is greater than or equal to every symbol that appears in the top row of $T$, add $a$ to the right of the top
row of $T$. Otherwise, let $C$ be the leftmost column whose topmost symbol is strictly greater than $a$. Slide column
$C$ down by one space and add $a$ as a new entry on top of $C$. Output the new tableau.
\end{algorithm}

An \defterm{rPS tableau} is a finite array of symbols of $\aA$ in which columns are top-aligned, the entries in the top
row are increasing from left to right, and the entries in each column are non-decreasing from top to bottom. For
example,
\begin{equation}
\label{eq:rpstabeg}
\tikz[tableau]\matrix{
1 \& 2 \& 5 \& 6 \\
1 \& 3 \& 5 \\
1 \& 6 \\
3  \\
};
\end{equation}
is an rPS tableau. Again, the insertion algorithm is very straightforward:

\begin{algorithm}[{\cite[\S~3.2]{thomas_longest}}]
\label{alg:rpsinsertone}
\textit{Input:} An rPS tableau $T$ and a symbol $a \in \aA$.

If $a$ is greater than every symbol that appears in the top row of $T$, add $a$ to the right of the top
row of $T$. Otherwise, let $C$ be the leftmost column whose topmost symbol is greater than or equal to $a$. Slide column
$C$ down by one space and add $a$ as a new entry on top of $C$. Output the new tableau.
\end{algorithm}

\subsection{Monoids from insertion}

Let $\mathsf{M} \in \set{\plac,\hypo,\sylv,\sylvsharp,\baxt,\stal,\allowbreak\taig,\lps,\rps}$ and $u \in \aA^*$. Using the insertion algorithms
described above, one can compute from $u$ a combinatorial object $\plit_{\mathsf{M}}(u)$ of the type associated to
$\mathsf{M}$ in \fullref{Table}{tbl:objects}. If $\mathsf{M} \neq \baxt$, one computes $\plit_{\mathsf{M}}(u)$ by
starting with the empty combinatorial object and inserting the symbols of $u$ one-by-one using the appropriate insertion
algorithm and proceeding through the word $u$ either left-to-right or right-to-left as shown in the table. For
$\mathsf{M} = \baxt$, one uses a slightly different procedure: $\pbaxt{u}$ is the pair of twin binary search trees
$\parens[\big]{\psylvsharp{u},\psylv{u}}$.

\afterpage{%
  \clearpage\upshape
\begin{landscape}
\begin{table}[ht]
  \centering
  \caption{Insertion algorithms used to compute combinatorial objects, and the corresponding monoids. `L-to-R' and
    `R-to-L' abbreviate `left-to-right' and `right-to-left', respectively. The algorithm used to compute $\pbaxt{\cdot}$
    is a special case and is discussed in the main text.}
  \label{tbl:objects}
  \begin{tabular}{llllll}
    \toprule
    \textit{Monoid}        & \textit{Symbol} & \textit{Combinatorial object}   & \textit{Algorithm}           & \textit{Direction} & \textit{Example}                               \\
    \midrule
    Plactic                & $\plac$         & Young tableau                   & \ref{alg:placinsertone}      & L-to-R             & $\pplac{3613151265}$ is $\eqref{eq:tableaueg}$  \\
    Hypoplactic            & $\hypo$         & Quasi-ribbon tableau            & \ref{alg:hypoinsertone}      & L-to-R             & $\phypo{3613151265}$ is $\eqref{eq:qrteg}$      \\
    Sylvester              & $\sylv$         & Right strict binary search tree & \ref{alg:sylvinsertone}      & R-to-L             & $\psylv{3613151265}$ is $\eqref{eq:bsteg}$     \\
    $\#$-sylvester         & $\sylvsharp$    & Left strict binary search tree  & \ref{alg:sylvsharpinsertone} & L-to-R             & $\psylv{3613151265}$ is $\eqref{eq:bsteg}$     \\
    Baxter                 & $\baxt$         & Pair of twin BSTs               & ---                          & ---                & $\pbaxt{3613151265}$ is $\eqref{eq:twinbsteg}$   \\
    Stalactic              & $\stal$         & Stalactic tableau               & \ref{alg:stalinsertone}      & R-to-L             & $\pstal{3613151265}$ is $\eqref{eq:staltabeg}$ \\
    Taiga                  & $\taig$         & BST with multiplicities         & \ref{alg:taiginsertone}      & R-to-L             & $\ptaig{3613151265}$ is $\eqref{eq:bstmulteg}$  \\
    Left patience sorting  & $\lps$          & lPS tableau                     & \ref{alg:lpsinsertone}       & L-to-R             & $\plps{3613151265}$ is $\eqref{eq:lpstabeg}$    \\
    Right patience sorting & $\rps$          & rPS tableau                     & \ref{alg:rpsinsertone}       & L-to-R             & $\prps{3613151265}$ is $\eqref{eq:rpstabeg}$    \\
    \bottomrule
  \end{tabular}
\end{table}
\end{landscape}
}

For each $\mathsf{M} \in \set{\plac,\hypo,\sylv,\sylvsharp,\baxt,\stal,\taig,\lps,\rps}$, define the relation $\equiv_{\mathsf{M}}$ by
\[
u \equiv_{\mathsf{M}} v \iff \plit_{\mathsf{M}}(u) = \plit_{\mathsf{M}}(v).
\]
In each case, the relation $\equiv_{\mathsf{M}}$ is a congruence on $\aA^*$, and so the factor monoid
$\mathsf{M} = \aA^*\!/{\equiv_{\mathsf{M}}}$ can be formed, and is named as in \fullref{Table}{tbl:objects}. The
rank-$n$ analogue is the factor monoid $\mathsf{M}_n = \aA_n^*/{\equiv_{\mathsf{M}}}$, where the relation
  $\equiv_{\mathsf{M}}$ is naturally restricted to $\aA_n^* \times \aA_n^*$.

It follows from the definition of $\equiv_{\mathsf{M}}$ (for any
$\mathsf{M} \in \set{\plac,\hypo,\sylv,\sylvsharp,\allowbreak\baxt,\stal,\taig,\lps,\rps}$) that each element $[u]_{\equiv_{\mathsf{M}}}$ of
the factor monoid $\mathsf{M}$ can be identified with the combinatorial object $\plit_{\mathrm{M}}(u)$.

The \defterm{evaluation} (also called the \defterm{content}) of a word $u \in \aA^*$, denoted $\ev{u}$, is the infinite
tuple of non-negative integers, indexed by $\aA$, whose $a$-th element is $|u|_a$; thus this tuple describes the number
of each symbol in $\aA$ that appears in $u$. It is immediate from the definition of the monoids above that if
$u \equiv_{\mathsf{M}} v$, then $\ev{u} = \ev{v}$, and hence it makes sense to define the evaluation of an element $p$
of one of these monoids to be the evaluation of any word representing it.

Note that $\mathsf{M}_n$ is a submonoid of $\mathsf{M}$ for all $n$, and that $\mathsf{M}_m$ is a submonoid of
$\mathsf{M}_n$ for all $m \leq n$. In each case, $\mathsf{M}_1$ is a free monogenic monoid and thus commutative, but
that $\mathsf{M}_n$ is non-commutative for $n \geq 2$ (and thus $\mathsf{M}$ is also non-commutative).

\section{Identities}

\fullref{Subsections}{subsec:hypoiden} to \ref{subsec:baxtiden} find shortest non-trivial identities satisfied by
$\hypo$, $\stal$, $\taig$, $\sylv$, $\sylvsharp$, and $\baxt$, in that order. \fullref{Subsection}{subsec:lpsrpsiden}
discusses the situation for $\lps$ and $\rps$, and \fullref{Subsection}{subsec:placiden} discusses the current state of
knowledge for $\plac$.

\subsection{Preliminaries}

An \defterm{identity} over an alphabet $X$ is a formal equality $u = v$, where $u$ and $v$ are words in the free monoid
$X^*$, and is non-trivial if $u$ and $v$ are not equal. The elements of $X$ are called \defterm{variables}. A monoid $M$
\defterm{satisfies} the identity $u = v$ if, for every homomorphism $\phi : X^* \to M$, the equality $\phi(u) = \phi(v)$
holds in $M$. It is often useful to think of this in the informal way mentioned in the introduction: $M$ satisfies the
identity $u = v$ if equality in $M$ holds under every substitution of variables in the words $u$ and $v$ by elements of
$M$. Note that if a monoid satisfies an identity $u=v$, all of its submonoids and all of its homomorphic images satisfy
$u=v$.

\begin{lemma}
  \label{lem:shortest}
  Let $M$ be a monoid satisfying a non-trivial identity. Then the shortest \textparens{in terms of the sums of the
    lengths of the two sides} non-trivial identity satisfied by $M$ is an identity over an alphabet with at most two
  variables.
\end{lemma}

\begin{proof}
  It is sufficient to show that, for any non-trivial identity over an alphabet with at most three variables, there is a
  strictly shorter identity with at most two variables. So suppose that $u = v$ is a non-trivial identity over $X$
  satisfied by $M$, where $|X| \geq 3$. Interchanging $u$ and $v$ if necessary, assume $|u| \leq |v|$. Let $e_M$ be the
  identity (that is, the multiplicative neutral element) of $M$.

  First consider the case where $u$ is a prefix of $v$. By non-triviality, $u$ is strictly shorter than $v$, so that
  $v = uw$ for some non-empty word $w$. Let $x \in X$ be a variable that appears in $w$. So $x^{|u|_x} = x^{|v|_x}$ is a
  non-trivial identity. To see that it is satisfied by $M$, proceed as follows: Let $\phi' : \set{x}^* \to M$ be a
  homomorphism. Extend $\phi'$ to a homomorphism $\phi : X \to M$ by defining $\phi(z) = e_M$ for all
  $z \in X \setminus \set{x}$. Note that $\phi'(x^{|u|_x}) = \phi(u)$ and $\phi'(x^{|v|_x}) = \phi(v)$ by the definition
  of $\phi'$, and that $\phi'(u) = \phi'(v)$ since $M$ satisfies $u = v$. Combining these equaltities shows that
  $\phi(x^{|u|_x}) = \phi(x^{|v|_x})$. Hence $M$ satisfies the non-trivial identity $x^{|u|_x} = x^{|v|_x}$ over the
  alphabet $\set{x}$.

  Now suppose $u$ is not a prefix of $v$. Suppose $u = u_1\cdots u_{|u|}$ and $v = v_1\cdots v_{|v|}$. Let $i$ be
  minimal such that the variables $u_i$ and $v_i$ are different; suppose these variables are $x$ and $y$. Let $u'$ and
  $v'$ be obtained from $u$ and $v$ respectively by deleting all variables other than $x$ and $y$. By the choice of $i$,
  $u' = v'$ is a non-trivial identity. To see that it is satisfied by $M$,
proceed as follows: Let
  $\phi' : \set{x,y}^* \to M$ be a homomorphism. Extend $\phi'$ to a homomorphism $\phi : X^* \to M$ by defining
  $\phi(z) = e_M$ for all $z \in X \setminus \set{x,y}$. Note that $\phi(u) = \phi'(u')$ and $\phi(v) = \phi'(v')$ by
  the definitions of the words $u'$ and $v'$ and the homomorphism $\phi$; note also that $\phi(u) = \phi(v)$ since $M$
  satisfies the identity $u = v$.  Combining these equalities shows that
$\phi'(u')= \phi'(v')$. So $M$ satisfies the
  non-trivial identity $u' = v'$ over the alphabet $\set{x,y}$.

  Note that in both cases the resulting identity is strictly shorter than
the original one.
\end{proof}

\begin{lemma}
  \label{lem:lengths}
  Let $M$ be a monoid that contains a free monogenic submonoid and suppose that $M$ satisfies an identity $u = v$ over
  an alphabet $X$. Then $|u| = |v|$ and $|u|_x = |v|_x$ for all $x \in X$. Consequently, if $u = v$ is non-trivial, then
  $|X| \geq 2$.
\end{lemma}

\begin{proof}
  Let $a \in M$ generate a free monogenic submonoid of $M$ and let $e_M$ be the identity of $M$.  Let $x \in X$. Define
  \[
    \phi_x : X^* \to M,\qquad x \mapsto a,\quad z \mapsto e_M \text{ for all $z \in X \setminus \set{x}$.}
  \]
  Then $a^{|u|_x} = \phi_x(u) = \phi_x(v) = a^{|v|_x}$ since $M$ satisfies $u = v$. Since $a$ generates a free
  monogenic submonoid, $|u|_x = |v|_x$. Since this holds for all $x \in X$, it follows that $|u| = |v|$.
\end{proof}

The \defterm{length} of an identity $u = v$ with $|u| = |v|$ is defined to be $|u| = |v|$; note that this applies to any
identity satisfied by a monoid that contains a free monogenic submonoid by \fullref{Lemma}{lem:lengths}.

Two identities $u = v$ and $u' = v'$ are \defterm{equivalent} if one can be obtained from the other by possibly renaming
variables and swapping the two sides of the equality. For example, $xyxy = yxxy$ and $xyyx = yxyx$ are equivalent, since
the second can be obtained from the first by interchanging the variables $x$ and $y$ and swapping the two sides.

\begin{lemma}
  \label{lem:lengthtwothree}
  \begin{enumerate}
  \item Up to equivalence, the only length-$2$ non-trivial identity that can be satisfied by a monoid that contains a free monogenic
    submonoid is $xy = yx$; thus any homogeneous monoid satisfying a length-$2$ identity is commutative.
  \item Up to equivalence, the only length-$3$ non-trivial identities over a two-letter alphabet that can be satisfied
    by a monoid that contains a free monogenic submonoid are
    \[
      xxy =xyx,\qquad xxy = yxx,\qquad xyx = yxx.
    \]
  \end{enumerate}
\end{lemma}

\begin{proof}
  \begin{enumerate}
  \item Let $u_1u_2 = v_1v_2$ be a non-trivial identity satisfied  by some
monoid that contains a free monogenic submonoid. Suppose $u_1$ and $v_1$ are
the same variable. Then \fullref{Lemma}{lem:lengths}, implies that $u_2$ and
    $v_2$ are also the same variable, contradicting non-triviality. Thus $u_1$ and $v_1$ are different variables. Then
    \fullref{Lemma}{lem:lengths}, implies that $u_2$ is the same variable as $v_1$ and $v_2$ is the same variable as
    $u_1$. Thus, up to equivalence, the identity is $xy = yx$.
  \item Let $u_1u_2u_3 = v_1v_2v_3$ (where $u_i,v_i \in \set{x,y}$) be a non-trivial identity satisfied by some
    homogeneous monoid. Let $I$ be the set of indices $i \in \set{1,2,3}$ where the variables $u_i$ and $v_i$ are the
    same. Clearly, non-triviality shows that $|I| < 3$. If $|I| = 2$, then, interchanging $x$ and $y$ if necessary,
    there is a unique $i \in \set{1,2,3}$ where $u_i$ is $x$ and $v_i$ is $y$.
Then $|u|_x = 1 + |v|_x$, since the words
    $u$ and $v$ differ only at the $i$-th symbol; this contradicts \fullref{Lemma}{lem:lengths}. If $|I| = 0$, then
    every symbol of $u$ differs from every symbol of $v$. Since $|u| = 3$, at least one of $|u|_x$ and $|u|_y$ is at
    least $2$. Interchanging $x$ and $y$ if necessary, assume $|u|_x \geq 2$. Then $|v|_x = |u|_y \leq 1$, which is a
    contadiction.

    The only remaining possibility is $|I| = 1$. Interchanging $x$ and $y$ if necessary, assume that $u_i$ and $v_i$ are
    both $x$ for exactly one $i \in {1,2,3}$. For each of the three possibilities, the conditions $|u|_x = |v|_x$ and
    $|u|_y = |v|_y$ require the other symbols in $u$ must be $x$ and $y$, with the corresponding symbols in $v$ being the
    opposite. Each possibility gives one of the given identities. \qedhere
\end{enumerate}
\end{proof}

Let $\mathsf{M} \in \set{\plac,\hypo,\sylv,\taig,\stal,\baxt,\lps,\rps}$. Then
$\mathsf{M}$ contains the free monogenic
submonoid $\mathsf{M}_1$.
(Note that $\mathsf{M}_1$ thus satisfies the non-trivial identity $xy = yx$.) By
\fullref{Lemma}{lem:lengths}, if $\mathsf{M}$ satisfies a non-trivial identity, it must be over an alphabet with at
least two variables. The aim is to find shortest non-trivial identities satisfies by each monoid $\mathsf{M}$:
\fullref{Lemmata}{lem:shortest} and \ref{lem:lengths} show that it will suffice to consider identities $u = v$ over
$\set{x,y}$ with $|u| = |v|$. Note that since none of the monoids $\mathsf{M}$ is commutative, any identity must have
length at least $3$.

\subsection{Hypoplactic monoid}
\label{subsec:hypoiden}

The authors proved the following result using a quasi-crystal structure for the hypoplactic monoid
\cite[Theorem~9.3]{cm_hypoplactic}; this subsection presents a direct proof.

\begin{proposition}
  \label{prop:hypoiden}
  The hypoplactic monoid satisfies the following non-trivial identities
  \begin{align*}
    xyxy = xyyx &= yxxy = yxyx; \\
    xxyx &= xyxx.
  \end{align*}
  Furthermore, up to equivalence, these are the shortest non-trivial identities satisfied by the hypoplactic monoid.
\end{proposition}

\begin{proof}
  The first goal is to show that these identities are satisfied by $\hypo$. Let
$s,t \in \hypo$, and let
  $p,q \in \aA^*$ be words representing $s,t$, respectively. Let $B = \set{a_1
< \ldots < a_k}$ be the set of symbols in
  $\aA$ that appear in at least one of $p$ and $q$.

  By \fullref{Algorithm}{alg:hypoinsertone}, symbols $a_{i+1}$ and $a_i$ are on the same row of $\phypo{w}$
  (for any $w \in B^*$) if and only if the word $w$ does not contain a symbol $a_i$ somewhere to the right of a symbol
  $a_{i+1}$ (since inserting this symbol $a_i$ results in the part of the quasi-ribbon tableau that contains $a_{i+1}$
  being glued \emph{below} the entry $a_i$).

  Suppose $pqpq$ contains a symbol $a_i$ somewhere to the right of a symbol $a_{i+1}$. Then, regardless of whether these
  symbols both lie in $u$, both lie in $v$, or one lies in $u$ and the other lies in $v$, the word $pqqp$ also contains
  a symbol $a_i$ to the right of a symbol $a_{i+1}$. Similar reasoning establishes that if any of the words $pqpq$,
  $qppq$, $qpqp$, $pqqp$, $ppqp$, or $pqpp$, contains a symbol $a_i$ somewhere to the right of a symbol $a_{i+1}$, so do
  all the others. Hence either the symbols $a_i$ and $a_{i+1}$ are on the same row in all of $\phypo{pqpq}$,
  $\phypo{pqqp}$, $\phypo{qppq}$, $\phypo{qpqp}$, $\phypo{ppqp}$, and $\phypo{pqpp}$, or on different rows in all of them.

  A quasi-ribbon tableau is clearly determined by its evaluation and by knowledge of whether adjacent different entries
  are on the same row. Thus, noting that $\ev{pqpq} = \ev{qppq} = \ev{qpqp} = \ev{pqqp}$, it follows from the previous
  paragraph that
  \begin{multline*}
    stst = \phypo{pqpq} = tsst = \phypo{qppq} \\ = tsts = \phypo{qpqp} = stts = \phypo{pqqp}.
  \end{multline*}
  Similarly, noting that $\ev{ppqp} = \ev{pqpp}$, it follows that
  \[
    ssts = \phypo{ppqp} = \phypo{pqpp} = stss.
  \]

  The next goal is to show that, up to equivalence, these are the only length-$4$ non-trivial identities satisfied by
  $\hypo$.  Let $u = v$ be a length-$4$ non-trivial identity over $\set{x,y}$ satisfied by $\hypo$. By
  \fullref{Lemma}{lem:lengths}, $|u|_x = |v|_x$ and $|u|_y = |v|_y$.

  By \fullref{Algorithm}{alg:hypoinsertone}, symbols $1$ and $2$ are on different rows of $\phypo{w}$ (for any
  $w \in \set{1,2}^*$) if and only if the word $w$ contain a symbol $1$ somewhere to the right of a symbol $2$. Suppose
  $u$ contains a factor $xy$. Let $\phi : \set{x,y}^* \to \hypo$ map $x$ to $2$ and $y$ to $1$; then $\phypo{\phi(u)}$
  has $1$ and $2$ on different rows. The same must hold for $\phypo{\phi(v)}$, and thus $v$ must contain a factor
  $xy$. The converse is similar, and the reasoning is parallel for factors $yx$. Thus $u$ contains a factor $xy$
  (respectively, $yx$) if and only if $v$ contains a factor $xy$ (respectively, $yx$).

  If both $u$ and $v$ contain only factors $xy$, then $u = x^{|u|_x}y^{|u|_y} = x^{|v|_x}y^{|v|_y} = v$, which
  contradicts non-triviality. Similarly, it is impossible for $u$ and $v$ to contain only factors $yx$. So both $u$ and
  $v$ contain factors $xy$ and $yx$.

  Interchanging $x$ and $y$ if necessary, assume $|u|_x = |v|_x \geq |u|_y = |v|_y$. First consider the case where
  $|u|_x = |v|_x = 3$ and $|u|_y = |v|_y = 1$. Since $u$ and $v$ both contain
$xy$ and $yx$, this implies that
  $u,v \in \set{xxyx,xyxx}$. Now consider the case $|u|_x = |v|_x = 2$ and
$|u|_y = |v|_y = 2$. Again, since $u$ and $v$
  both contain $xy$ and $yx$, this implies that $u,v \in
\set{xyxy,xyyx,yxxy,yxyx}$. Thus $u = v$ is equivalent to one
  of the identities in the statement.

  The final goal is to prove that no length-$3$ non-trivial identity is
satisfied by $\hypo$. Consider the length-$3$
  identities from \fullref[(2)]{Lemma}{lem:lengthtwothree}:
  \[
    xxy =xyx,\qquad xxy = yxx,\qquad xyx = yxx.
  \]
  If one puts $x = 1$ and $y = 2$ into the first two identities and $x = 2$ and $y = 1$ in the third, one sees that
  $\hypo$ satisfies none of them:
  \begin{align*}
    \phypo{112} = \tableau{1 \& 1 \& 2 \\} &\neq \tableau{1 \& 1 \\ \& 2 \\} = \phypo{121}; \\
    \phypo{112} = \tableau{1 \& 1 \& 2 \\} &\neq \tableau{1 \& 1 \\ \& 2 \\} = \phypo{211}; \\
    \phypo{212} = \tableau{1 \\ 2 \& 2 \\} &\neq \tableau{1 \& 2 \& 2 \\} = \phypo{122}.
  \end{align*}
  Thus $\hypo$ satisfies no length-$3$ identity.
\end{proof}

\subsection{Stalactic and taiga monoids}
\label{subsec:staltaigiden}

\begin{lemma}
  \label{lem:staliden}
  The stalactic monoid satisfies the non-trivial identity $xyx = yxx$.
\end{lemma}

\begin{proof}
  Let $x,y \in \stal$ and let $u,v \in \aA^*$ be words representing $x,y$, respectively. Since the tree $\pstal{w}$ is
  computed by applying \fullref{Algorithm}{alg:stalinsertone} to each symbol in the word $w \in \aA^*$, proceeding right
  to left, it is clear that the order of the rightmost appearances of each symbol in $w$ determines order of symbols in
  the top row of the stalactic tableau $\pstal{w}$. Since the order of rightmost appearances of each symbol in the words
  $uvu$ and $vuu$ are the same, the order of symbols in the top row of
$\pstal{uvu}$ and $\pstal{vuu}$ is also the same. Since
  $\ev{uvu} = \ev{vuu}$, the length of corresponding columns in $\pstal{uvu}$ and $\pstal{vuu}$ are equal. Hence
  $xyx = \pstal{uvu} = \pstal{vuu} = yxx$.
\end{proof}

\begin{lemma}
  \label{lem:taigiden}
  The taiga monoid does not satisfy either of the identities $xxy = xyx$ or $xxy = yxx$, and does not satisfy a
  non-trivial identity of length $2$.
\end{lemma}

\begin{proof}
  To see that $\taig$ does not satisfy either of the identities $xxy = xyx$ or $xxy = yxx$, note that the rightmost
  symbol of $w \in \aA^*$ labels the root node of $\ptaig{w}$, but the rightmost symbols of both these identities are
  different. Thus substituting $1$ for $x$ and $2$ for $y$ shows that neither identity can be satisfied by $\taig$.

  Finally, $\taig$ does not satisfy a non-trivial identity of length~$2$ since it is clearly non-commutative.
\end{proof}

Since the taiga monoid is a homomorphic image of the stalactic monoid \cite[\S~5]{priez_lattice}, any identity that is satisfied by $\stal$ is
satisfied by $\taig$. Combining this with \fullref{Lemmata}{lem:lengthtwothree}, \ref{lem:staliden}, and \ref{lem:taigiden}
gives the following results:

\begin{proposition}
  \label{prop:staliden}
  The stalactic monoid satisfies the non-trivial identity $xyx = yxx$. Furthermore, this is the unique shortest non-trivial identity
  satisfied by the stalactic monoid.
\end{proposition}

\begin{proposition}
  \label{prop:taigiden}
  The taiga monoid satisfies the non-trivial identity $xyx = yxx$. Furthermore, this is the unique shortest non-trivial identity
  satisfied by the taiga monoid.
\end{proposition}

\subsection{Sylvester monoid}
\label{subsec:sylviden}

By the definition of a binary search tree, if a tree has a node with label $a$, then any other node with label $a$ is
either above it or in its left subtree. This shows that there is a unique path from the root to a leaf node that
contains all nodes labelled $a$. In particular, there is at most one leaf node with label $a$. Furthermore, there is a
unique node $a$ that is both the leftmost node with label $a$ and the node labelled $a$ that is furthest from the
root. Similarly, there is a unique node $a$ that is both the rightmost node with label $a$ and the node labelled $a$
that is closest to the root.

\begin{lemma}
  \label{lem:sylvleftcanc}
  The sylvester monoid is left-cancellative.
\end{lemma}

\begin{proof}
  Let $u,v \in \aA^*$ and $a \in \aA$. Suppose that $\psylv{au} = \psylv{av}$. That is, the binary search trees
  $\psylv{au}$ and $\psylv{av}$ are equal. Note further that the last symbol $a$ inserted into both binary search trees
  is a leaf node. Deleting this (unique) leaf node labelled $a$ from the equal trees $\psylv{au}$ and $\psylv{av}$
  leaves equal binary search trees. That is, $\psylv{u}$ and $\psylv{v}$ are equal.

  Since $u$ and $v$ are arbitrary words and $a$ represents an arbitrary generator, it follows that $\sylv$ is
  left-cancellative.
\end{proof}

\begin{lemma}
  \label{lem:sylvuxxvyx}
  The sylvester monoid does not satisfy an identity equivalent to one of the
form $uxx = vyx$, for any $u,v\in\{x,y\}^*$.
\end{lemma}

\begin{proof}
  Using the \fullref{Algorithm}{alg:sylvinsertone}, one sees that
  \[
    \psylv{\cdots 22} =
    \begin{tikzpicture}[tinybst,baseline=(0)]
      \node (root) {$2$}
      child { node (0) {$2$}
        child { node[triangle] (00) {} }
        child[missing]
      }
      child { node[triangle] (1) {} };
    \end{tikzpicture}
    \text{ and }
    \psylv{\cdots 12} =
    \begin{tikzpicture}[tinybst,baseline=(0)]
      \node (root) {$2$}
      child { node (0) {$1$}
        child { node[triangle] (00) {} }
        child[missing]
      }
      child { node[triangle] (1) {} };
    \end{tikzpicture};
  \]
  note that these binary search trees differ in the left child of the root node. Thus to see that the identity
  $uxx = vyx$ cannot be satisfied by $\sylv$, one can substitute $2$ for $x$ and $1$ for $y$.
\end{proof}

\begin{lemma}
  \label{lem:sylvsameeval}
  Let $p,q,r \in \sylv$ be such that $\evlit(p) = \evlit(q) = \evlit(r)$. Then $pr = qr$.
\end{lemma}

\begin{proof}
  Let $u,v,w \in \aA^*$ be words representing $p,q,r$, respectively. The tree $\psylv{x}$ is computed by applying
  \fullref{Algorithm}{alg:sylvinsertone} to each symbol in $x$, proceeding right to left. That is, $\psylv{uw}$ and
  $\psylv{vw}$ are obtained by inserting $u$ and $v$ (respectively) into $\psylv{w}$. Since $\ev{u} = \ev{v} = \ev{w}$,
  every symbol that appears in $u$ or $v$ also appears in $w$ and thus in $\psylv{w}$.

  Consider how further symbols from $u$ or $v$ are inserted into $\psylv{w}$:
  \begin{itemize}
  \item Let $a$ be the smallest symbol that appears in $u$, $v$, and $w$. Then the leftmost node of $\psylv{w}$ must be
    labelled by $a$. Thus, during the computation of $\psylv{uw}$ and $\psylv{vw}$, all symbols $a$ in $u$ and $v$ will
    certainly be inserted into the left subtree of this leftmost node in $\psylv{w}$, and no other symbols are inserted
    into this left subtree by the minimality of $a$.
  \item Let $c$ be some symbol other than $a$ that appears in $u$, $v$, and $w$, and let $b$ be the maximum symbol less
    than $c$ appearing in $u$, $v$, and $w$. Let $N_c$ be the leftmost node in $\psylv{w}$ labelled by $c$ and let $N_b$
    be the rightmost node in $\psylv{w}$ labelled by $b$.

    Suppose $v$ is the label of the lowest common ancestor $N_v$ of the nodes
$N_b$ and $N_c$. If $N_v$ is neither $N_b$
    not $N_c$, then $b \leq v < c$, which implies $v = b$ by the choice of $b$, which contradicts the fact that $N_b$ is
    the rightmost node labelled $b$. Thus the lowest common ancestor of $N_b$
and $N_c$ must be one of $N_b$ or
    $N_c$. That is, one of the following must hold:
    \begin{itemize}
    \item $N_b$ is above $N_c$. Then $N_c$ is in the right subtree of $N_b$, since $b < c$. Since there is no node that
      is to the right of $N_b$ and to the left of $N_c$, it follows that $N_c$ has an empty left subtree. Thus any
      symbol $c$ in $u$ or $v$ will be inserted into this currently empty left subtree of $N_c$, and no other symbols
      will be inserted into this subtree by the maximality of $b$ among the symbols less than $c$.
    \item $N_c$ is above $N_b$. Then $N_b$ is in the left subtree of $N_c$, since $b < c$. Since there is no node that
      is to the right of $N_b$ and to the left of $N_c$, it follows that $N_b$ is the left child of $N_c$, and that
      $N_b$ has an empty right subtree. Thus any symbol $c$ in $u$ or $v$ will be inserted into this currently empty
      right subtree of $N_b$, and no other symbols will be inserted into this subtree by the maximality of $b$ among the
      symbols less than $c$.
    \end{itemize}
  \end{itemize}
  Combining these cases, one sees that every symbol $d$ from $u$ or $v$ is inserted into a particular previously empty
  subtree of $\psylv{w}$, dependent only on the value of the symbol $d$ (and not on its position in $u$ or $v$), and
  that unequal symbols are inserted into different subtrees. Since $\ev{u} = \ev{v}$, the same number of symbols $d$ are
  inserted, for each such symbol $d$. Hence $\psylv{uw} = \psylv{vw}$ and thus $pr = \pstal{uw} = \pstal{vw} = qr$.
\end{proof}

\begin{proposition}
  \label{prop:sylviden}
  The sylvester monoid satisfies the non-trivial identity $xyxy = yxxy$. Furthermore, up to equivalence, this is the unique shortest
  identity satisfied by the sylvester monoid.
\end{proposition}

\begin{proof}
  Let $x,y \in \sylv$. Let $p = r = xy$ and $q = yx$. Then $\ev{p} = \ev{q} = \ev{r}$, so $pr = qr$ by
  \fullref{Lemma}{lem:sylvsameeval}. Thus $\sylv$ satisfies $xyxy = yxxy$.

  Since $\hypo$ is a homomorphic image of $\sylv$ \cite[Example~6]{priez_lattice}, and $\hypo$ satisfies no length-$3$
  identity by \fullref{Proposition}{prop:hypoiden}, it follows that $\sylv$ cannot satisfy a length-$3$ identity.

  So let $u = v$ be some length-$4$ identity satisfied by $\sylv$. Suppose $u = u_1u_2u_3u_4$ and $v = v_1v_2v_3v_4$,
  where $u_i,v_i \in \set{x,y}$. Since the rightmost symbol in a word $w$ determines the root node of $\psylv{w}$, it
  follows that $u_4$ and $v_4$ are the same symbol; interchanging $x$ and $y$ if necessary, assume this is $y$. If $u_1$
  and $v_1$ were the same symbol, then the left-cancellativity of $\sylv$ (\fullref{Lemma}{lem:sylvleftcanc}) would
  imply that $\sylv$ satisfied the length-$3$ identity $u_2u_3u_4 = v_2v_3v_4$,
which contradicts the previous
  paragraph. Thus $u_1$ and $v_1$ are different symbols. Interchanging $u$ and $v$ if necessary, assume that $u_1$ is
  $x$ and $v_1$ is $y$. The condition $|u|_y = |v|_y$ implies that at least one of $u_2$ and $u_3$
  is $y$, which yields the following three possibilities for $u$:
  \[
    xyyy, \quad xxyy, \quad xyxy.
  \]
  The conditions $|u|_x = |v|_x$ and $|u|_y = |v|_y$ now imply the following four possibilities for the identity $u = v$
  (with the first possibility for $u$ above giving two different identities):
  \[
    xyyy = yxyy,\quad xyyy = yyxy, \quad xxyy = yxxy,\quad xyxy = yxxy.
  \]
  The last of these is the identity already proven to hold in $\sylv$. The second and third cannot hold in $\sylv$ by
  \fullref{Lemma}{lem:sylvuxxvyx}. In the first identity, putting $x = 1$ and $y = 2$ proves that it is not satisfied
  by $\sylv$:
  \[
    \psylv{1222} = \tikz[microbst,baseline=(0)] \node (root) {$2$} child { node (0) {$2$} child { node (00) {$2$} child { node (000) {$1$} } child[missing] } child[missing] } child[missing];
    \neq \tikz[microbst,baseline=(0)] \node (root) {$2$} child { node (0) {$2$} child { node (00) {$1$} child[missing] child { node (001) {$2$} } } child[missing] } child[missing]; = \psylv{2122}
  \]
  Thus $xyxy = yxxy$ is the unique shortest identity satisfied by $\sylv$.
\end{proof}

Since the sylvester and $\#$-sylvester monoids are anti-isomorphic, the following results follow by dual reasoning:

\begin{lemma}
  \label{lem:sylvsharprightcanc}
  The $\#$-sylvester monoid is right-cancellative.
\end{lemma}

\begin{lemma}
  \label{lem:sylvsharpxxuxyv}
  The $\#$-sylvester monoid does not satisfy an identity equivalent to one of
the form $xxu = xyv$, for any $u,v\in\{x,y\}^*$.
\end{lemma}

\begin{lemma}
  \label{lem:sylvsharpsameeval}
  Let $p,q,s \in \sylvsharp$ be such that $\evlit(p) = \evlit(q) = \evlit(s)$. Then $sp = sq$.
\end{lemma}

\begin{proposition}
  \label{prop:sylvsharpiden}
  The $\#$-sylvester monoid satisfies the non-trivial identity $yxyx = yxxy$. Furthermore, up to equivalence, this is the unique shortest
  identity satisfied by the $\#$-sylvester monoid.
\end{proposition}

\subsection{Baxter monoid}
\label{subsec:baxtiden}

\begin{lemma}
  \label{lem:baxtsameeval}
  Let $p,q,r,s \in \baxt$ be such that $\ev{p} = \ev{q} = \ev{r} = \ev{s}$. Then $spr = sqr$.
\end{lemma}

\begin{proof}
  Let $t,u,v,w \in \aA^*$ represent $p,q,r,s$, respectively.

  Let $t' = wt$ and $u' = wu$; note that $\ev{t'} =
  \ev{u'}$. By \fullref{Lemma}{lem:sylvsameeval} applied to the elements $\psylv{t'},\psylv{u'},\psylv{v}$, it follows
  that $\psylv{t'v} = \psylv{u'v}$; thus $\psylv{wtv} = \psylv{wuv}$

  Let $t'' = tv$ and $u'' = uv$; note that $\ev{t''} =
  \ev{u''}$. By \fullref{Lemma}{lem:sylvsharpsameeval} applied to the elements $\psylvsharp{t'},\psylvsharp{u'},\psylvsharp{w}$, it follows
  that $\psylvsharp{wt''} = \psylvsharp{wu''}$; thus $\psylvsharp{wtv} = \psylvsharp{wuv}$.

  Hence
  \begin{align*}
    spr = \pbaxt{wuv} &= \parens[\big]{\psylvsharp{wtv},\psylv{wtv}} \\
    &= \parens[\big]{\psylvsharp{wuv},\psylv{wuv}} = \pbaxt{wuv} = sqp. \qedhere
  \end{align*}
\end{proof}

\begin{proposition}
  \label{prop:baxtiden}
  The Baxter monoid satisfies the identities
  \[
    yxxyxy = yxyxxy\text{ and }xyxyxy = xyyxxy.
  \]
  Furthermore, up to equivalence, these are the unique shortest non-trivial identities satisfied by the Baxter monoid.
\end{proposition}

\begin{proof}
  Let $x,y \in \baxt$. Let $p = r = xy$ and $q = s = yx$. Then $\ev{p} = \ev{q} = \ev{r} = \ev{s}$. Thus, by
  \fullref{Lemma}{lem:baxtsameeval}, $yxxyxy = spr = sqr = yxyxxy$. The same
reasoning with $s = xy$ shows that
  $xyxyxy = xyyxxy$.

  The next step is to show that, up to equivalence, these are the only identities of length $6$ satisfied by $\baxt$. So
  suppose that $u = v$ is an identity of length $6$ satisfied by $\baxt$, with $u = u_1u_2\cdots u_6$ and
  $v = v_1v_2\cdots v_6$, where $u_i,v_i \in \set{x,y}$.  Since the first symbol in a word $w$ determines the root
  symbol of the left-hand tree in the pair $\pbaxt{u}$, and since the last symbol determines the root symbol of the
  right-hand tree, it follows that $u_1$ and $v_1$ must be the same variable, and that $u_6$ and $v_6$ must be the same
  variable. By \fullref{Lemmata}{lem:sylvuxxvyx} and \ref{lem:sylvsharpxxuxyv}, $u_2$ and $v_2$ must be the same
  variable, and $u_5$ and $v_5$ must be the same variable. Since $\sylv$ and $\sylvsharp$ are both homomorphic images of
  $\baxt$ \cite[Proposition~3.7]{giraudo_baxter2}, the identity $u=v$ is also satified by $\sylv$ and
  $\sylvsharp$. Since $\sylv$ is left-cancellative by \fullref{Lemma}{lem:sylvleftcanc}, $u_3u_4u_5u_6 = v_3v_4v_5v_6$
  is satisfied by $\sylv$ and is thus equivalent to $xyxy = yxxy$. The aim is to characterize $u = v$ up to equivalence,
  so assume that $u_3u_4u_5u_6 = v_3v_4v_5v_6$ actually is the identity $xyxy = yxxy$.

  Since $\sylvsharp$ is right-cancellative by \fullref{Lemma}{lem:sylvsharprightcanc}, $u_1u_2u_3u_4 = v_1v_2v_3v_4$ is
  satisfied by $\sylvsharp$ and is thus equivalent to $yxyx = yxxy$, which is equivalent (by swapping the two sides) to
  $yxxy = yxyx$ and (by interchaning $x$ and $y$) to $xyxy = xyyx$. Combining these with $u_3u_4u_5u_6 = v_3v_4v_5v_6$
  from the previous paragraph shows that $u = v$ is (up to equivalence) either $yxxyxy = yxyxxy$ or $xyxyxy = xyyxxy$.

  Finally, it is necessary to show that no length-$5$ identity is satisfied by $\baxt$. Suppose $u = v$ is an identity
  of length $5$ satisfied by the Baxter monoid. Suppose that $u = u_1u_2\cdots u_5$ and $v = v_1v_2\cdots v_5$, where
  $u_i,v_i \in \set{x,y}$. As before, considering the root symbols of the left-hand and right-hand trees in $\pbaxt{w}$
  shows that $u_1$ and $v_1$ must be the same symbol, and that $u_5$ and $v_5$ must be the same symbol.

  Since $\sylv$ and $\sylvsharp$ are both homomorphic images of $\baxt$ \cite[Proposition~3.7]{giraudo_baxter2}, both of
  them satisfy the identity $u = v$. Furthermore, $\sylvsharp$ is right-cancellative and so satisfies the identity
  $u_1u_2u_3u_4 = v_1v_2v_3v_4$; while $\sylv$ is left-cancellative and so satisfies the identity
  $u_2u_3u_4u_5 = v_2v_3v_4v_5$. By \fullref{Proposition}{prop:sylviden}, the unique length-$4$ identity satisfied by
  $\sylv$ is $xyxy = yxxy$, so the identity $u = v$ is either $xxyxy = xyxxy$ or $yxyxy = yyxxy$. Deleting the rightmost
  symbol $y$ from each of these identities yields $xxyx = xyxx$ and $yxyx = yyxx$, one of which must be the identity
  $u_1u_2u_3u_4 = v_1v_2v_3v_4$ satisfied by $\sylvsharp$. This is a contradiction, since by
  \fullref{Proposition}{prop:sylvsharpiden} the only length-$4$ identity satisfied by $\sylvsharp$ is $yxyx = yxxy$.
\end{proof}

\subsection{Left and right patience sorting monoids}
\label{subsec:lpsrpsiden}

The present authors and Silva \cite[\S~4.2]{cms_patience1} have shown that the
elements $\plps{21}$ and $\plps{1}$
generate a free submonoid of $\lps$.  Since free monoids of rank at least $2$
satisfy no non-trivial identities, it follows that $\lps$ does not
satisfy a non-trivial identity. Furthermore, $\lps_n$ for $n \geq 2$
satisfies no non-trivial identities, while $\lps_1$, as a monogenic monoid, is of course commutative and satisfies the
identity $xy = yx$.

For any $n$, the monoid $\rps_n$ satisfies the identity $(xy)^{n+1}=(xy)^nyx$ but does not satisfy any non-trivial
identity of length less than or equal to $n$ \cite[\S~4.2]{cms_patience1}, which immediately implies that $\rps$ does not
satisfy any non-trivial identity.

\subsection{Plactic monoid}
\label{subsec:placiden}

Recently, the present authors together with Klein, Kubat, and Okni\'{n}ski used the combinatorics of Young tableaux to
prove that $\plac_n$ does not satisfy any identity of length less than or equal to $n$ \cite[Proposition~3.1]{ckkmo_placticidentity},
which implies that $\plac$ does not satisfy a non-trivial identity.

For finite-rank plactic monoids, some more partial results are known. First, $\plac_1$ is monogenic and thus commutative and
satisfies $xy = yx$. Kubat \& Okni\'{n}ski \cite{kubat_identities} have shown that $\plac_2$ satisfies Adian's identity
$xyyxxyxyyx = xyyxyxxyyx$, and that $\plac_3$ satisfies the identity $pqqpqp = pqpqqp$, where $p(x,y)$ and $q(x,y)$ are
respectively the left and right side of Adian's identity (and so the identity $pqqpqp = pqpqqp$ has sixty variables $x$
or $y$ on each side). Furthermore, $\plac_3$ does not satisfy Adian's identity \cite[p.~111--2]{kubat_identities}.

Furthermore $\plac_3$ satisfies the identity $pqppq = pqqpq$, where as above $p(x,y)$ and $q(x,y)$ are respectively the
left and right side of Adian's identity \cite[Corollary~5.4]{ckkmo_placticidentity}. The proof technique was rather
different, via an embedding of $\plac_3$ into the direct product of two copies of the monoid of $3\times 3$
upper-triangular tropical matrices, which in turn is shown to satisfy the given identity.

\begin{conjecture}
  For each $n \geq 4$, there is a non-trivial identity satisfied by $\plac_n$.
\end{conjecture}

\bibliography{\jobname}
\bibliographystyle{alphaabbrv}

\end{document}